\documentclass{article}
\usepackage[utf8]{inputenc}
\usepackage{amsmath}
\usepackage{amsthm}
\usepackage{amsfonts}
\usepackage{amssymb}
\usepackage{latexsym}
\usepackage{graphicx}
\date{}
\graphicspath{ {Images/} }
\title{Subelliptic Resolvent Estimates for Non-self-adjoint Semiclassical Schr\"{o}dinger Operators}
\author{Ben Bellis\\Department of Mathematics\\UCLA\\Los Angeles, CA, 90095\\bbellis@math.ucla.edu}
\renewcommand{\(}{\left(}
\renewcommand{\)}{\right)}
\newtheorem{theorem}{Theorem}
\newtheorem{lemma}{Lemma}
\newtheorem{example}{Example}
\newtheorem*{remark}{Remark}
\newtheorem{cor}{Corollary}
\newcommand{\R}{\mathbb{R}}
\newcommand{\C}{\mathbb{C}}
\newcommand{\Sch}{\mathcal{S}}
\newcommand{\rr}{{\rm Re}\,}
\newcommand{\im}{{\rm Im}\,}
\begin{document}

\maketitle
\begin{abstract}
    In this paper we prove a subelliptic resolvent estimate for a broad class of semiclassical non-self-adjoint Schr\"{o}dinger
    operators with complex potentials when the spectral parameter is in a parabolic neighborhood of the imaginary axis. 
\end{abstract}

\section{Introduction}

Non-self-adjoint Schr\"{o}dinger operators can appear in a variety of settings. These settings can range physical problems to purely mathematical ones.
Such examples include the study of the Ginzburg-Landau equation in superconductivity \cite{ahelf}, \cite{henry}, the Orr-Somerfeld operator in fluid dynamics \cite{redp}, \cite{shk}, the theory of scattering resonances \cite{zw2}, or non-self-adjoint perturbations of self-adjoint operators \cite{h-s}. 
In the self-adjoint case, the spectral theorem provides a powerful tool to control the resolvent of Schr\"{o}dinger operators.
However, there is no suitable analog to this for non-self-adjoint operators. 

In this paper, we study semiclassical non-self-adjoint differential operators, and are thus concerned with the behavior of the 
resolvent as the semiclassical parameter $h$ tends towards $0$.
The general difficulty is that for non-self-adjoint semiclassical operators the spectrum does not control the resolvent, which may become very
large far away from the spectrum as $h\rightarrow 0$. 
By a theorem of Davies \cite{davies} and Dencker, Sj\"{o}strand, and Zworski \cite{dsz}, for a non-self-adjoint semiclassical Schr\"{o}dinger operator of the form
$P=-h^{2}\Delta+ V\(x\)$, for $V\in C^{\infty}\(\R^{n}\)$, and any $z$ of the
form $z=\xi_{0}^{2}+V\(x_{0}\)$ where $\(x_{0},\xi_{0}\)\in\R^{2n}$ and $\im \xi_{0}\cdot V'\(x_{0}\)\neq 0$,
$z$ is an ``almost eigenvalue" of $P$, in the sense that there exists a family of functions $u\(h\)\in L^{2}$ for which
$\|\(P-z\)u\(h\)\|_{L^2}=O\(h^{\infty}\)\|u\(h\)\|_{L^2}$. 
Thus, when $\rr V\geq 0$ we should not generally expect to have much control over the resolvent of such an operator in the interior of the right half-plane. 
So instead we will study resolvent estimates of such operators when the spectral parameter $z$ is near the boundary of this region.

In this paper we show that for a broad class of non-self-adjoint semiclassical Schr\"{o}dinger operators there is an unbounded parabolic region near the imaginary axis where the resolvent is well controlled. Let us now introduce the precise assumptions on our operators. 

Let $p \in C^{\infty}\(\R^{2n}\)$ be such that
\begin{equation}\label{eq:pdef}
p\(X\)=|\xi|^2 + V\(x\),
\end{equation} 
where $V=V_{1}+iV_{2}$ with $V_{1}, V_{2}$ real valued and $X=\(x,\xi\)$, with $x$, $\xi \in \R^{n}$.

We place the following conditions on the potential $V$:
\begin{equation}\label{eq:v1pos}V_{1}\(x\)\geq 0,\quad x\in\R^{n}\end{equation}
\begin{equation}\label{eq:vv'}|V_{2}\(x\)|\lesssim 1 + |V_{2}'\(x\)|^{2},\quad x\in \R^{n},\end{equation}
\begin{equation}\label{eq:vbd}\partial^{\alpha}V \in L^{\infty}\(\R^{n}\),\ |\alpha|\geq 2.\end{equation}
Here, and throughout the paper, we use the notation ``$f\lesssim g$'' to denote that there exists a constant $c>0$
such that $f\leq cg$. 
We define the Weyl quantization of a symbol $a\(x, \xi\)$ by
$$
a^{w}\(x, D_{x}\)u\(x\) = \int_{\mathbb{R}^{2n}} e^{2\pi i\(x-y\)\cdot\xi}a\(\frac{x+y}{2}, \xi\) u\(y\) dy d\xi
$$
and the semiclassical Weyl quantization by
$$
a^{w}\(x, hD_{x}\)u\(x\) = \int_{\mathbb{R}^{2n}} e^{2\pi i\(x-y\)\cdot\xi}a\(\frac{x+y}{2}, h\xi\) u\(y\) dy d\xi,
$$
where $0<h\leq 1$. Note that 
$$p^{w} = -\frac{h^{2}}{4\pi^{2}}\Delta + V\(x\).$$
We first prove the following a priori estimate for this operator.
\begin{theorem}\label{thm1}
For such $p$, let $T\geq 0$ be such that
\begin{equation}\label{eq:vv't}
|V_{2}\(x\)|-T\lesssim |V_{2}'\(x\)|^{2},\quad x\in \R^{n},
\end{equation}
and choose any $K\in \R$, $K>1$. Then there exist positive constants $h_{0}$, $A$, and $M$ such that for all $0<h<h_{0}$,
$z\in\textbf{C}$ with $|z|\geq KT+Mh$ and ${\rm Re}\, z\leq A h^{2/3} \(|z|-T\)^{1/3}$, and $u\in \mathcal{S}$,
$$
\left\|\(p^{w}\(x,hD_{x}\)-z\)u\right\|_{L^2}\gtrsim h^{2/3}\(|z|- T\)^{1/3}\|u\|_{L^2}.
$$
\end{theorem}
We then use this to get a resolvent estimate on $L^2$.
\begin{theorem}\label{thm2}
For $p$ as above, $P$, the $L^2$-graph closure of $p^{w}\(x,hD_{x}\)$ on $\Sch$ is the maximal realization of $p^{w}\(x,hD_{x}\)$
    equipped with the domain $D_{max}=\left\{u\in L^{2} : p^{w}u\in L^{2}\right\}$. For $T$, $h$ and $z$ as above we have the resolvent estimate
    $$
    \left\| \(P-z\)^{-1} \right\|_{L^{2}\rightarrow L^{2}} \lesssim h^{-2/3}\(|z|- T\)^{-1/3}.
    $$
\end{theorem}
\begin{remark}
For such $P$, we have that $P$ is accretive because
\[\rr\(Pu,u\)_{L^2}=\(\(-\frac{h^{2}}{4\pi^{2}}\Delta + V_{1}\)u,u\)_{L^{2}}\geq 0,\quad u\in D_{max}.\]
Thus Theorem \ref{thm2} implies that $P$ is maximally accretive.
\end{remark}
\begin{figure}[h]
\centering
\includegraphics[width=.4\textwidth]{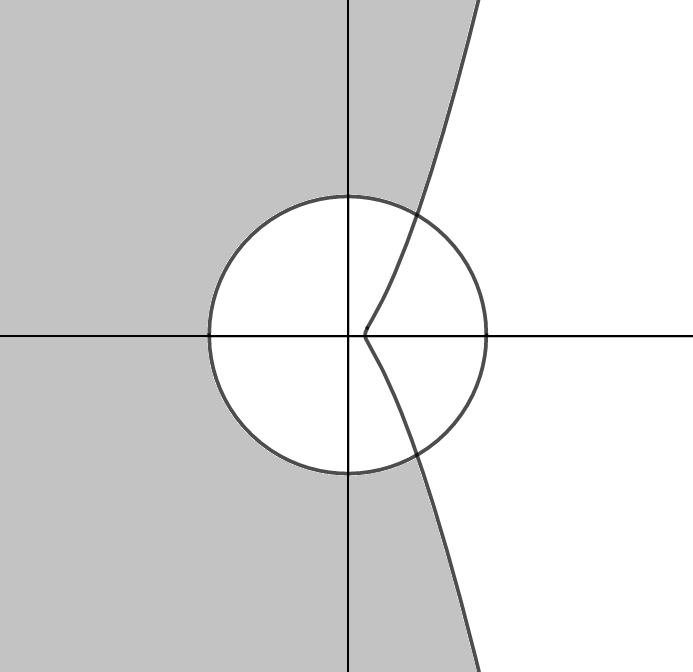}
\caption{The shaded region indicates the values of $z$ for which the Theorems \ref{thm1} and \ref{thm2} apply.}
\label{zregion}
\end{figure}

Similar resolvent estimates have been attained for different classes of semiclassical non-self-adjoint operators.
Herau, Sj\"{o}strand, and Stolk proved a similar resolvent estimate for the Kramers-Fokker-Planck operator under certain conditions \cite{hss}. We use a multiplier method inspired by one used in \cite{hss}, but our proof proceeds quite differently. Theirs uses the FBI transform in a compact region of phase space and and Weyl-H\"{o}rmander calculus with a suitable metric near infinity, while ours works globally using the Wick quantization and some standard Weyl calculus. Hitrik and Sj\"{o}strand attained a similar estimate for certain one-dimensional non-self-adjoint Schr\"{o}dinger operators \cite{h-s}, with ellipticity assumptions on the potential. 
Also, Dencker, Sj\"{o}strand, and Zworski showed that for non-self-adjoint semiclassical operators, under suitable assumptions including ellipticity at infinity, the resolvent can be similarly estimated
in a small region near a boundary point of the range of the symbol, away from critical values of the symbol \cite{dsz}. 
What distinguishes our result, in addition to the relatively direct proof, is that we have fairly loose conditions on the potential, with no requirement of ellipticity, and
we attain a resolvent estimate for $z$ in an unbounded region.

To demonstrate the applicability of this result, here are some examples of cases where it can be used.
\begin{example}

Let $V\(x\)=q\(x\)$ for $q$ any quadratic form with $\rr q\(x\)\geq 0$. By diagonalization we can see that
$|\im q\(x\)|\lesssim |\im q'\(x\)|^{2}$, and $q''$ is constant so we can apply the above theorems to $p=|\xi|^{2}+ q\(x\)$ with $T=0$. 
Thus for some $h_{0}$, $A$, and $M$,
$$\left\|\(-\frac{h^{2}}{4\pi^{2}}\Delta + q\(x\) - z\)^{-1}\right\|_{L^2\rightarrow L^2}\lesssim h^{-2/3}|z|^{-1/3},$$
for all $z\in\textbf{C}$ with $|z|>Mh$ and ${\rm Re}\, z\leq A h^{2/3}|z|^{1/3}$ and $0<h\leq h_{0}$.

\end{example}
We can apply these theorems to many other classes of potentials. Note that the condition $|V_{2}\(x\)|-T\lesssim |V_{2}'\(x\)|^{2}$ implies that $T$ will be at least as large as the maximum absolute value of a critical value of $V_{2}$. 
\begin{example}

Let $V\in C^{\infty}\(\R^{2}\)$ be given by $V\(x_{1},x_{2}\)=ix_{1}^2 + isin\(x_{2}\)$.
Then $|V\(x_{1},x_{2}\)|-1 \lesssim |V'\(x_{1},x_{2}\)|^2$ so applying the above to $p=|\xi|^{2}+V$ with $T=1$ and any $K>1$ yields, for some $h_{0}$, $A$, and $M$,
$$\left\|\(-\frac{h^{2}}{4\pi^{2}}\Delta + i\(x_{1}^{2}+sin\(x_{2}\)\)-z\)^{-1}u\right\|_{L^2}\lesssim h^{-2/3}\(|z|-1\)^{-1/3},$$ 
for all $z\in\textbf{C}$ with $|z|>K + Mh$ and ${\rm Re}\, z\leq A h^{2/3}\(|z| - 1\)^{1/3}$ and $0<h\leq h_{0}$.
\end{example}
For a broader example we also have the following:
\begin{example}
Let $V_{2}\in C^{\infty}\(\R^{n};\R\)$ be a Morse function with finitely many critical points that satisfies \eqref{eq:vbd}. Furthermore 
suppose that $|V_{2}'\(x\)|\gtrsim |x|$ for all $x\in\R^{n}$ with $|x| > R$ for some $R > 0$. Let $x_{1},... x_{N}\in\R^{n}$ be the critical points of V,
and let $T=\max\limits_{1\leq j\leq N}|V_{2}\(x_{j}\)|$. Since
$V_{2}$ is Morse, in a neighborhood of each $x_j$, $V_{2}\(x\)=V_{2}\(x_{j}\)+q_{j}\(x-x_{j}\)+O\(|x-x_{j}|^{3}\)$ 
for some nondegenerate quadratic form $q_{j}$.
So $V_{2}'\(x\)= q'_{j}\(x-x_{j}\) + O\(|x-x_{j}|^{2}\)$ and $|q'_{j}\(x-x_{j}\)|\sim |x-x_{j}|$.
Then, locally near $x_{j}$ we have
$$|V_{2}\(x\)|-T\leq |q_{j}\(x-x_{j}\)|+O\(|x-x_{j}|^{3}\)$$
$$\lesssim |x-x_{j}|^{2}\lesssim |V_{2}'\(x\)|^{2}.$$
Thus $|V_{2}\(x\)|-T\lesssim |V_{2}'\(x\)|^{2}$ in a neighborhood of each critical point.
For $x$ away from critical points and $|x|\leq R$, $|V_{2}'\(x\)|$ is bounded below away from 0 and $|V_{2}\(x\)|$ is bounded above, so $|V_{2}\(x\)|-T\lesssim |V_{2}'\(x\)|^{2}$ here as well.
Lastly \eqref{eq:vbd} implies that $|V_{2}\(x\)|\lesssim 1 + |x|^{2}$ so $|V_{2}\(x\)|\lesssim |V_{2}'\(x\)|^{2}$ for $|x|>R$, and we see that the preceding theorems can be applied to $p=|\xi|^{2} +V_{1}\(x\) + iV_{2}\(x\)$ for any such $V_{2}$ and any $V_{1}$ satisfying \eqref{eq:v1pos} and \eqref{eq:vbd}.
\end{example}

The plan of the paper is as follows. In Section 2 we will construct a bounded weight function $g$ to be used in proving Theorem \ref{thm1}. 
Then in Section 3 we will provide a brief overview of the Wick quantization. In Section 4 we prove Theorem \ref{thm1} by using 
the weight function as a bounded multiplier to prove an
estimate for the Wick quantization of $p$ and use the relationship between the Wick and Weyl quantizations as well as some Weyl symbol calculus to
get the desired estimate. In Section 5 we prove Theorem \ref{thm2} by showing the estimate from Theorem \ref{thm1} can be extended to the maximal
domain of $P$.
In Section 6, we show how the preceding proofs can be modified to prove a similar result for a larger class of potential functions if we additionally require that $|z|$ be bounded above.

\section{The Weight Function}
Let 
$$\lambda\(X\):=\(|\xi|^{2}+V_{1}\(x\)+|V'_{2}\(x\)|^{2}\)^{1/2}.$$
It is worth noting that for this $p$ we have that
$$
\lambda\(X\)^{2} \lesssim {\rm Re}\, p + H^{2}_{{\rm Im}\, p}{\rm Re}\, p \lesssim \lambda\(X\)^{2}, 
$$
because this motivates our choice of weight function. 
Here, for $f\in C^1\(\mathbb{R}^{2n}\)$, we use the notation $H_f$ to denote the Hamiltonian vector field of $f$, i.e. given 
$f\(x,\xi\),\, g\(x,\xi\)\in C^{1}\(\mathbb{R}^{2n}\)$ we define
$$
H_{f}g= \left\{ f,g \right\} = 
\partial_{\xi}f\cdot\partial_{x}g - \partial_{x}f\cdot\partial_{\xi}g.
$$

\begin{lemma}\label{lemma1}
Let $p \in C^{\infty}\(\R^{2n}\)$ be given by $p\(x,\xi\)=|\xi|^{2}+V\(x\)$ 
with $V=V_{1}+iV_{2}$, $V_{1},V_{2}$ real valued, $V''\in L^{\infty}$, and $V_{1}\geq 0$.
Let $\psi\in C_{c}^{\infty}\(\R;[0,1]\)$ be a cutoff with $\psi\(t\)=1$ for $|t|\leq 1$ and $\psi\(t\)=0$ for $|t|\geq 2$. 

There exist $0<\epsilon<1$ and  $0< h_{0}\leq 1$ depending on $p$ 
such that for all $0< h\leq h_{0}$ and $X$ with $\lambda\(X\)\geq h^{1/2}$, the smooth weight 
function $G$ given by 
$$G\(X\) =\epsilon h^{-1/3}\frac{H_{\im p}\rr p}{\lambda\(X\)^{4/3}}\psi\(\frac{4\rr p}{\(h \lambda\(X\)\)^{2/3}}\)$$
satisfies 
\begin{equation}\label{eq:con1}|G\(X\)|=O\(\epsilon \), \end{equation}
\begin{equation}\label{eq:con2}|G'\(X\)| = O\(\epsilon h^{-1/2}\),\end{equation}
and 
\begin{equation}\label{eq:ineq}{\rm Re}\,  p\(X\) + hH_{{\rm Im}\, p} G\(X\) \gtrsim h^{2/3}\lambda\(X\)^{2/3}.\end{equation}
\end{lemma}
\begin{proof}
The support of $G$ is contained in the region where $|\xi|^{2}\leq\frac{1}{2}\(h \lambda\(X\)\)^{2/3}$, so we see that since $\psi \leq 1$ we have 

\[|G\(X\)| \leq \epsilon h^{-1/3}\frac{2|V_{2}'\(x\)||\xi|}{\lambda\(X\)^{4/3}}\psi\(\frac{\rr p\(X\)}{\(h \lambda\(X\)\)^{2/3}}\)\]
\begin{equation}
\lesssim \epsilon h^{-1/3}\frac{\lambda\(X\)\(h \lambda\(X\)\)^{1/3}}{\lambda\(X\)^{4/3}}
\lesssim \epsilon,  
\end{equation}
which verifies that $G$ satisfies \eqref{eq:con1}. 
Note that as $V_{1}''\in L^{\infty}$ and $V_{1}\geq 0$ we have, using a standard inequality (Lemma 4.31 of \cite{zw}), that 
\begin{equation}\label{eq:v1'v1}
\left|V_{1}'\(x\)\right|\lesssim V_{1}\(x\)^{1/2}.
\end{equation}
This and \eqref{eq:vbd} then imply that
\begin{equation}\label{eq:lambda'}
\partial^{\alpha}\lambda^{2}=O\(\lambda\),\quad |\alpha|= 1.
\end{equation}

Now, to check \eqref{eq:con2}, 
one can use \eqref{eq:vbd}, \eqref{eq:v1'v1}, \eqref{eq:lambda'}, 
and the fact that $|\xi|\lesssim\(h \lambda\(X\)\)^{1/3}$ on the support of $G$ to get the following estimates on the support of $G$:
\begin{equation}\label{eq:est1} \left| \frac{H_{V_{2}}|\xi|^{2}}{\lambda\(X\)^{4/3}}\right| = O\(h^{1/3}\), \end{equation}
\begin{equation}\label{eq:est2}\left|\partial^{\alpha} \frac{H_{V_{2}}|\xi|^{2}}{\lambda\(X\)^{4/3}}\right| = 
O\(\lambda\(X\)^{-1/3}\) = O\(h^{-1/6}\) ,\quad  |\alpha|=1,\end{equation}
\begin{equation}\label{eq:est4}\left|\partial^{\alpha}\(\psi\(\frac{4\(|\xi|^2+V_{1}\(x\)\)}{\(h \lambda\(X\)\)^{2/3}}\)\)\right| = 
O\(\frac{|\xi|+|V_{1}'\(x\)|}{\(h\lambda\(X\)\)^{2/3}} + \lambda\(X\)^{-1}\) \end{equation}
\[= O\(\(h\lambda\(X\)\)^{-1/3}+\lambda\(X\)^{-1}\) = O\(h^{-1/2}\),\quad  |\alpha|=1.\]
Thus by \eqref{eq:est1}, \eqref{eq:est2}, and \eqref{eq:est4},
$$|G'\(X\)| \lesssim \epsilon h^{-1/3}\( O\(h^{-1/6}\) + O\(h^{1/3} h^{-1/2}\) \) = O\(\epsilon h^{-1/2}\),$$
which verifies \eqref{eq:con2}. 

Now we shall attain \eqref{eq:ineq} in the case where $|\xi|^{2}+V_{1}\(x\)
\leq\frac{1}{4}\(h \lambda\(X\)\)^{2/3}\leq \frac{1}{4}\lambda\(X\)^2$, 
and so $|V_{2}'\(x\)|^{2}\geq \frac{3}{4}\lambda\(X\)^2$.
In this region $\psi\(\frac{4 \rr p}{\(h\lambda\(X\)\)^{2/3}}\)\equiv 1$, and so $G\(X\) =\epsilon h^{-1/3}\frac{H_{V_2}|\xi|^{2}}{\lambda\(X\)^{4/3}}$.
Now we get
\begin{equation}\label{eq:hgv}
H_{V_2}G=\epsilon h^{-1/3}\(\frac{2|V_{2}'\(x\)|^{2}}{\lambda\(X\)^{4/3}}- \frac{8\(V_{2}'\(x\)\cdot \xi\)^{2}}{3 \lambda\(X\)^{10/3}}\).
\end{equation}
Thus 
$$
{\rm Re}\,  p\(X\) + h H_{{\rm Im}\, p} G\(X\) = \rr p\(X\) +\epsilon h^{2/3}\(\frac{2|V_{2}'\(x\)|^{2}}{\lambda\(X\)^{4/3}} 
-\frac{8\(V_{2}'\(x\)\cdot\xi\)^{2}}{3 \lambda\(X\)^{10/3}}\)$$
$$\geq \rr p\(X\) + \epsilon h^{2/3}\(\frac{2|V_{2}'\(x\)|^2}{\lambda\(X\)^{4/3}} - \frac{2|V_{2}'\(x\)|^{2}}{3 \lambda\(X\)^{4/3}}\)
\geq \epsilon h^{2/3}\frac{4|V_{2}'\(x\)|^2}{3 \lambda\(X\)^{4/3}}$$
$$\gtrsim \epsilon h^{2/3}\left|V_{2}'\(x\)\right|^{2}\lambda\(X\)^{-4/3}\gtrsim \epsilon h^{2/3}\lambda\(X\)^{2/3}.$$
It remains to show the bound in the region where $|\xi|^2 + V_{1}\(x\) \geq \frac{1}{4}\(h \lambda\(X\)\)^{2/3}$.
Using \eqref{eq:est1}, \eqref{eq:est2}, and \eqref{eq:est4} we get that 
\begin{multline*}
|hH_{V_2}G|\leq\epsilon h^{2/3}\lambda\(X\) O\(\lambda\(X\)^{-1/3} \) \\
+ \epsilon h^{2/3}\lambda\(X\)O\(h^{1/3}\(\(h \lambda\(X\)\)^{-1/3} + \lambda\(X\)^{-1}\)\)\\
=O\(\epsilon \(h\lambda\(X\)\)^{2/3}\).
\end{multline*}
Here, fixing $\epsilon$ sufficiently small yields
$$|\xi|^{2} +V_{1}\(x\) + hH_{V_2}G \gtrsim h^{2/3}\lambda\(X\)^{2/3} - O\(\epsilon  h^{2/3}\lambda\(X\)^{2/3}\) \gtrsim  h^{2/3}\lambda\(X\)^{2/3}.$$
This completes the proof of the lemma.
\end{proof}
\begin{cor}\label{cor1}
For such $p$ as above, there exists a bounded real weight function $g\in C^{\infty}\(\R^{2n}\)$ and constants $C_{0}, h_{0}>0$ such that for
all $0<h\leq h_{0}$ and all $X\in \R^{2n}$
$|g\(X\)|\leq 1$, $|g'\(X\)|=O\(h^{-1/2}\)$ and
\begin{equation}\label{eq:sub}
\rr p\(X\) + h H_{\im p}g\(X\)+C_{0}h\gtrsim h^{2/3}\lambda\(X\)^{2/3}.
\end{equation}
\end{cor}
\begin{proof}
Let $G$ be a weight function for $p$ as constructed in Lemma \ref{lemma1}, and set $\epsilon$ small enough that $|G|\leq 1$.
Now we extend $G$ to all of $\R^{2n}$ by defining 
$$g\(X\) = \(1-\psi\(\frac{2 \lambda\(X\)^2}{h}\)\)G\(X\),$$
where $\psi\in C_{c}^{\infty}\(\R;[0,1]\)$, $\psi\(t\)=1$ for $|t|\leq 1$, $\psi\(t\)=0$ for $|t|\geq 2$, as before. 
By \eqref{eq:con2} and \eqref{eq:lambda'}, 
\begin{equation}\label{eq:ngh}
|g'| \lesssim \frac{ \lambda\(X\)}{h}\left|\psi'\(\frac{2 \lambda\(X\)^2}{h}\)\right|\left|G\(X\)\right| + 
\(1-\psi\(\frac{2 \lambda\(X\)^2}{h}\)\)|G'\(X\)|\end{equation}
$$\lesssim h^{-1/2}.$$
By Lemma \ref{lemma1}, \eqref{eq:sub} holds in the region where $\lambda\(X\)>h^{1/2}$ 
for $h$ sufficiently small since $g=G$ there.
When $\lambda\(X\)<\frac{1}{2}h^{1/2}$ we have $H_{V_{2}}g\(X\)=0$ and 
$h^{2/3}\lambda\(X\)^{2/3}<h$ so the inequality holds
in this region as well.
When $\frac{1}{2}h^{1/2}\leq \lambda\(X\)\leq h^{1/2}$, using \eqref{eq:ngh} we get
$$
|\xi|^{2} +V_{1}\(x\) +h H_{V_{2}}g\(X\)+C_{0}h\geq C_{0}h - O\(h^{1/2}\lambda\(X\)\) \gtrsim h^{2/3}\lambda\(X\)^{2/3},
$$
for $C_{0}$ sufficiently large.
\end{proof}
\section{Wick quantization overview}
Before proving Theorem \ref{thm1} we first will note some facts about the Wick quantization.
For $Y = \(y, \eta\) \in \mathbb{R}^{2n}$ and $x\in \mathbb{R}^{n}$ define 
$$\phi_{Y}\(x\)= 2^{n/4} e^{-\pi |x-y|^2} e^{2\pi i \eta\cdot\(x-y\)}.$$
Then for  $u\in L^{2}{\(\mathbb{R}^n\)}$ define the wave packet transform of $u$ by
$$W u\(Y\) = \(u, \phi_{Y}\),$$
where $\(\cdot , \cdot \)$ denotes the $L^{2}$ scalar product.
As proven in \cite{wicketc}, $W$ is an isometry from $L^{2}\(\mathbb{R}^n\)$ to $L^{2}\(\mathbb{R}^{2n}\)$
and continuous from $\mathcal{S}\(\mathbb{R}^{n}\)$ to $\mathcal{S}{\(\mathbb{R}^{2n}\)}$.
The function $\phi_Y$ is $L^2$ normalized, so the rank-one orthogonal projection of 
$u$ onto $\phi_{Y}$ is given by
$$
\Pi_{Y}u = \(u, \phi_{Y}\)\phi_{Y}.
$$
For a symbol $a\(x,\xi\)\in L^{\infty}\(\mathbb{R}^{2n}\)$
the Wick quantization of $a$ is given by 
\begin{equation}\label{eq:wickdef}
a^{Wick} = W^{*}a^{\mu}W,
\end{equation}
where $a^{\mu}$ denotes multiplication by $a$ and $W^*:L^{2}\(\R^{2n}\)\rightarrow L^{2}\(\R^{n}\)$ is the adjoint of $W$, or equivalently
$$
a^{Wick} = \int_{\mathbb{R}^{2n}}a\(Y\)\Pi_{Y} dY.
$$
We can see from \eqref{eq:wickdef} that for $a\in L^{\infty}\(\mathbb{R}^{2n}\)$ 
then $a^{Wick}$ is a bounded operator on $L^{2}\(\mathbb{R}^{n}\)$ with 
\begin{equation}\label{eq:l2bd}
\|a^{Wick}\|_{L^{2}\rightarrow L^{2}} \leq \|a\|_{L^{\infty}}    
\end{equation}
and that 
\begin{equation}\label{eq:wadj}
\(a^{Wick}\)^{*}=\(\overline{a}\)^{Wick}.
\end{equation}
More generally we can define the Wick quantization for symbols in the space of tempered distributions,
$\mathcal{S}'\(\mathbb{R}^{2n}\)$.
For $a\in\mathcal{S}'\(\mathbb{R}^{2n}\)$, $a^{Wick}$ is a map from $\mathcal{S}\(\mathbb{R}^{n}\)$
to $\mathcal{S}'\(\mathbb{R}^{n}\)$ defined by
$$
a^{Wick}u\(\overline{v}\)=a\(Wu \overline{Wv}\),
$$
for $u,v\in\mathcal{S}\(\mathbb{R}^{n}\)$.
As long as the symbol $a\in L^{\infty}_{loc}$ satisfies $\left|a\(X\)\right| \lesssim \(1+|X|\)^{N}$ for some $N$ then
$a^{\mu}$ is continuous as a map from $\mathcal{S}{\(\mathbb{R}^{2n}\)}$ 
to $L^{2}\(\mathbb{R}^{2n}\)$, and thus \eqref{eq:wickdef} implies that $a^{Wick}$ is continuous 
from $\mathcal{S}{\(\mathbb{R}^{n}\)}$ to $L^{2}\(\mathbb{R}^{n}\)$.
Furthermore, we have that for such symbols $a$ and $u\in \mathcal{S}\(\mathbb{R}^{n}\)$
\begin{equation}\label{eq:pos}a\geq 0 \Rightarrow \( a^{Wick}u,u\)_{L^2} \geq 0 .\end{equation}
Let $S\(m\)$ denote the symbol space
$$
S\(m\)=\left\{ f\in C^{\infty}\(\R^{2n}\) : \left|\partial^{\alpha}f\(X\)\right|\leq C_{\alpha}m\(X\),\, \forall \alpha\in \mathbb{N}^{2n} \right\},
$$
where $m$ is an order function on $\R^{2n}$ (cf. section 4.4 of \cite{zw}).
Another fact we will need from \cite{wicketc} is that for $a\in S\(m\)$, 
\begin{equation}\label{eq:ww}
a^{Wick}= a^{w} + r\(a\)^{w},
\end{equation}
where 
\begin{equation}\label{eq:remainder}
r\(a\)\(X\)=\int_{0}^{1}\int_{\mathbb{R}^{2n}}\(1-t\)a''\(X+tY\)Y^{2}e^{-2\pi|Y|^2}2^{n}dY dt.
\end{equation}
For smooth symbols $a$ and $b$ with
$a\in L^{\infty}\(\R^{2n}\)$ and $\partial^{\alpha}b \in L^{\infty}\(\R^{2n}\)$ for $|\alpha| = 2$ we have the following composition formula
proven in \cite{absource},
\begin{equation}\label{eq:comp}
a^{Wick}b^{Wick} = \(ab - \frac{1}{4\pi}a' \cdot b' + \frac{1}{4\pi i}\left\{ a, b \right\}\)^{Wick} + R, 
\end{equation}
where $\|R\|_{L^2 \rightarrow L^2} \lesssim \|a\|_{L^{\infty}} \sup\limits_{|\alpha|=2}\|\partial^{\alpha}b\|_{L^{\infty}} $.
We can see that the right-hand side is well defined as an operator $\mathcal{S}\(\mathbb{R}^{n}\)\rightarrow\mathcal{S}'\(\R^{n}\)$
because for $|\alpha_{1}|=|\alpha_{2}|=1$,
$$\(\partial^{\alpha_{1}}a\)\(\partial^{\alpha_{2}}b\) = \partial^{\alpha_{1}}\(a\partial^{\alpha_{2}}b\)
-a\(\partial^{\alpha_{1}+\alpha_{2}}b\).$$
As $|a\(X\)\partial^{\alpha_{2}}b\(X\)|\lesssim 1+|X|$ we can see that the symbol on the right-hand side of \eqref{eq:comp}
is indeed 
a tempered distribution.
\section{Proving the a priori estimate}
Now we will use the Wick quantization and the weight function from Lemma \ref{lemma1} to prove Theorem \ref{thm1}.

\begin{proof}[Proof of Theorem 1]
We will now follow a multiplier method based on section 4 of \cite{hps}. 
Let $g$ be a bounded real weight function for $p$ as constructed in Corollary \ref{cor1}.
We first note that for $u\in\mathcal{S}$, by \eqref{eq:wadj},
\begin{multline}\label{eq:c}
{\rm Re}\, \(\left[p\(\sqrt{h}X\)-z\right]^{Wick}u, \left[2-g\(\sqrt{h}X\)\right]^{Wick}u\) = \\
{\rm Re} \( \left[2-g\(\sqrt{h}X\)\right]^{Wick}\left[\, \(p\(\sqrt{h}X\)-z\)\right]^{Wick}u ,\, u\)= \\
 \({\rm Re}\( \left[2-g\(\sqrt{h}X\)\right]^{Wick}\left[\, \(p\(\sqrt{h}X\)-z\)\right]^{Wick}\)u ,\, u\).
\end{multline}
From \eqref{eq:wadj} it follows that 
$$
{\rm Re}\, a^{Wick}=\frac{1}{2}\(a^{Wick}+\(a^{Wick}\)^{*}\)
=\frac{1}{2}\(a^{Wick}+\(\overline{a}^{Wick}\)\)=\({\rm Re}\, a\)^{Wick}.
$$
Using this fact and the composition formula for the Wick quantization \eqref{eq:comp},
\begin{equation}\label{eq:d}
{\rm Re}\, \(\left[2-g\(\sqrt{h}X\)\right]^{Wick}\left[p\(\sqrt{h}X\)-z\right]^{Wick} \)= 
\end{equation}
\begin{multline*}
{\rm Re}\, \bigg[\(2-g\(\sqrt{h}X\)\)\(p\(\sqrt{h}X\)-z\) +
\frac{1}{4\pi}\nabla\(g\(\sqrt{h}X\)\)\cdot\nabla\(p\(\sqrt{h}X\)\)\\
- \frac{1}{4\pi i}\left\{g\(\sqrt{h}X\),p\(\sqrt{h}X\)\right\}\bigg]^{Wick} + S_{h}\\
=\bigg[\(2-g\(\sqrt{h}X\)\)\(\rr p\(\sqrt{h}X\) - {\rm Re}\, z\) \\
+ \frac{h}{4\pi}g'\(\sqrt{h}X\)\cdot \rr p'\(\sqrt{h}X\) + \frac{h}{4\pi}H_{V_{2}} g\(\sqrt{h}X\)\bigg]^{Wick} + S_{h},
\end{multline*}
where $\|S_{h}\|_{L^{2}\rightarrow L^{2}}=O\(h\)$.
Using \eqref{eq:v1'v1} and \eqref{eq:ngh} we have 
$$ \left|h g'\(\sqrt{h}X\)\cdot \rr p'\(\sqrt{h}X\) \right|\lesssim h^{1/2}\(\rr p\(\sqrt{h}X\)\)^{1/2}$$
$$
\lesssim rh + \frac{1}{r}\rr p\(\sqrt{h}X\),
$$
for arbitrary $r>0$. By taking $r$ large enough the $\frac{1}{r}\rr p\(\sqrt{h}X\)$ term can be absorbed by $\(2-g\(\sqrt{h}X\)\)\rr p\(\sqrt{h}X\)$.

Let 
\[y=|z|-T\geq \(K-1\)T+Mh . \]

By using \eqref{eq:sub} we get that for some $C_{1}$, $C_{2} > 0$ and arbitrary $A>0$,
$$
\(2-g\(\sqrt{h}X\)\)\(\rr p\(\sqrt{h}X\) -{\rm Re}\, z\) 
$$
$$+ \frac{h}{4\pi}g'\(\sqrt{h}X\)\cdot \rr p'\(\sqrt{h}X\) 
+ \frac{h}{4\pi}H_{V_{2}}g\(\sqrt{h}X\)
$$
\begin{equation}\label{eq:a}
\gtrsim \rr p\(\sqrt{h}X\) - 3 max\(0, {\rm Re}\, z\) + \frac{h}{4\pi}H_{V_{2}} g\(\sqrt{h}X\) + O\(h\)
\end{equation}
$$
\gtrsim h^{2/3}\lambda\(\sqrt{h}X\)^{2/3} - C_{1} max\(0, {\rm Re}\, z\) - C_{2}h
$$
$$
\geq h^{2/3}\(\lambda\(\sqrt{h}X\)^{2/3}- 2AC_{1}y^{1/3}\) + AC_{1} h^{2/3}y^{1/3}
$$
$$
+\ C_{1}\(A h^{2/3}y^{1/3} - max\(0, {\rm Re}\, z\)\) - C_{2}h.
$$
As we required that ${\rm Re}\, z \leq A h^{2/3}y^{1/3}$ we have that
\begin{equation}\label{eq:summary}
h^{2/3}\(\lambda\(\sqrt{h}X\)^{2/3} - 2AC_{1}y^{1/3}\) + C_{1}\(Ah^{2/3}y^{1/3} - max\(0, {\rm Re}\, z\)\)
\end{equation}
$$
\geq -2AC_{1}h^{2/3}y^{1/3}\psi\(\frac{B \lambda\(\sqrt{h}X\)^2}{y}\),
$$
where
\begin{equation}\label{eq:bdef}
B=\frac{1}{\(2AC_{1}\)^3},
\end{equation}
and $\psi$ is the same cutoff as before.
Fix the value of $A$ by choosing it small enough such that we can use that $|V_{2}\(x\)|-T\lesssim |V_{2}'\(x\)|^{2}$ to get
\begin{equation}\label{eq:bval}
 |p\(X\)|-T \leq \frac{B \lambda\(X\)^2}{4},\quad X\in\mathbb{R}^{2n}.
\end{equation}
Substituting \eqref{eq:summary} into \eqref{eq:a} gives
\begin{equation}\label{eq:b}
\(2-g\(\sqrt{h}X\)\)\(\rr p\(\sqrt{h}X\)-{\rm Re}\, z\) + \frac{h}{4\pi}g'\(\sqrt{h}X\)\cdot  \rr p'\(\sqrt{h}X\) 
\end{equation}
$$+ \frac{h}{4\pi}H_{V_{2}} g\(\sqrt{h}X\)$$
$$\gtrsim -2AC_{1}h^{2/3}y^{1/3}\psi\(\frac{B \lambda\(\sqrt{h}X\)^2}{y}\) - C_{2}h + AC_{1}h^{2/3}y^{1/3}.
$$
Now  \eqref{eq:pos}, \eqref{eq:c}, \eqref{eq:d}, and \eqref{eq:b} imply that, for $h$ sufficiently small and $\rr z\leq Ah^{2/3}y^{1/3}$,
$$
{\rm Re}\, \([p\(\sqrt{h}X\)-z]^{Wick}u, [2-g\(\sqrt{h}X\)]^{Wick}u\)+ h\|u\|^{2}_{L^2} +
$$
$$h^{2/3}y^{1/3}\(\psi\(\frac{B\lambda\(\sqrt{h}X\)^2}{y}\)^{Wick}u, u\)
\gtrsim h^{2/3}y^{1/3}\|u\|^{2}_{L^2}.
$$
By the Cauchy-Schwarz inequality and \eqref{eq:l2bd} we get that
$$
\left\|[p\(\sqrt{h}X\)-z]^{Wick}u\right\|_{L^2} + h\|u\|_{L^2}  +
h^{2/3}y^{1/3}\left\|\psi\(\frac{B\lambda\(\sqrt{h}X\)^2}{y}\)^{Wick}u\right\|_{L^2}
$$
$$\gtrsim h^{2/3}y^{1/3}\|u\|_{L^2}.
$$
Now we pick $M$ sufficiently large so that the $h\|u\|_{L^2}$ term can be absorbed by the right-hand side to get
\begin{equation}\label{eq:wickineq}
\left\|\left[p\(\sqrt{h}X\)-z\right]^{Wick}u\right\|_{L^2} + h^{2/3}y^{1/3}\left\|\psi\(\frac{B\lambda\(\sqrt{h}X\)^2}{y}\)^{Wick}u\right\|_{L^2}
\end{equation}
$$
\gtrsim h^{2/3}y^{1/3}\|u\|_{L^2}.
$$
This resembles the desired inequality, but we still need to switch from the Wick quantization to the Weyl quantization, and we need to 
deal with the term involving $\psi$. First we will switch to the Weyl quantization.
The Calder\'{o}n-Vaillancourt Theorem (Theorem 4.23 in \cite{zw}) states that for $a\in S\(1\)$ there exists
a universal constant $\lambda$ such that
\begin{equation}\label{eq:cv}
\left\|a^{w}\(x,D_{x}\)\right\|_{L^{2}\rightarrow L^{2}}\lesssim \sup\limits_{|\alpha|\leq \lambda n}\|\partial^{\alpha}a\|_{L^{\infty}}.
\end{equation}
From \eqref{eq:vbd} we have that 
$$
\partial^{\alpha}\(p\(\sqrt{h}X\)\)= O\(h^{|\alpha|/2}\), |\alpha|\geq 2,
$$
so we can apply the Calder\'{o}n-Vaillancourt theorem to the remainder term in \eqref{eq:ww} with $a\(X\)=p\(\sqrt{h}X\)-z$ to get
\begin{equation}\label{eq:pw2w}
\left\|p\(\sqrt{h}X\)^{Wick}u-zu\right\|_{L^2}=\left\|p\(\sqrt{h}X\)^{w}u-zu\right\|_{L^2} + O\(h\)\|u\|_{L^2}.
\end{equation}
To do the same thing to the other term on the left side of \eqref{eq:wickineq} we need to estimate the derivatives of $\psi\(\frac{B\lambda\(\sqrt{h}X\)^2}{y}\)$.
\begin{lemma}\label{lemma2}
\begin{equation}\label{eq:psisize}
\left|\partial^{\alpha}\(\psi\(\frac{B\lambda\(\sqrt{h}X\)^2}{y}\)\)\right|\lesssim
\frac{h^{1/2}}{y^{1/2}},\quad |\alpha|\geq 1.
\end{equation}
\end{lemma}
\begin{proof}
First, note that because $V''\in S\(1\)$ we have 
\begin{equation}\label{eq:lambda''}
\partial^{\alpha}\lambda\(X\)^{2}=\partial^{\alpha}\(|\xi|^{2}+V_{1}\(x\)+|V_{2}'\(x\)|^{2}\)\lesssim 1+|V_{2}'|\lesssim 1+\lambda,
\quad |\alpha|\geq 2.
\end{equation}
Also, for $X$ in the support of $\psi\(\frac{B\lambda\(\sqrt{h}X\)^2}{y}\)$ we have
\[\lambda\(\sqrt{h}X\)\lesssim y^{1/2},\]
and so, by \eqref{eq:lambda'}
\[\left|\partial^{\alpha}\(\frac{\lambda\(\sqrt{h}X\)^{2}}{y}\)\right|\lesssim \frac{h^{1/2}\lambda\(\sqrt{h}X\)}{y}
\lesssim \frac{h^{1/2}}{y^{1/2}},\quad |\alpha|=1,\]
and by \eqref{eq:lambda''}
\[\left|\partial^{\alpha}\(\frac{\lambda\(\sqrt{h}X\)^{2}}{y}\)\right|\lesssim \frac{h^{|\alpha|/2}\(1+\lambda\(\sqrt{h}X\)\)}{y}
\lesssim \frac{h}{y}+\frac{h}{y^{1/2}}\lesssim \frac{h^{1/2}}{y^{1/2}},\quad |\alpha|\geq 2.\]
We can express $\partial^{\alpha}\(\psi\(\frac{B\lambda\(\sqrt{h}X\)^2}{y}\)\)$ as a linear combination of terms of the form
\[\psi^{\(k\)}\(\frac{B\lambda\(\sqrt{h}X\)^2}{y}\)
\partial^{\gamma_{1}}\(\frac{\lambda\(\sqrt{h}X\)^{2}}{y}\) \ldots \partial^{\gamma_{k}}\(\frac{\lambda\(\sqrt{h}X\)^{2}}{y}\),\]
where $\alpha=\gamma_{1}+\ldots +\gamma_{k}$, $|\gamma_{i}|\geq 1$ for all $i$, $1\leq k\leq |\alpha|$.
Each such term is of size $O\(\(\frac{h}{y}\)^{k/2}\)$, proving the lemma.
\end{proof}

Using Lemma \ref{lemma2} and \eqref{eq:ww} we get
\begin{multline}\label{eq:psiw2w}
\left\|\psi\(\frac{B\lambda\(\sqrt{h}X\)^2}{y}\)^{Wick}u\right\|_{L^2}\\
=\left\|\psi\(\frac{B\lambda\(\sqrt{h}X\)^2}{y}\)^{w}u\right\|_{L^2} + O\(M^{-1/2}\)\|u\|_{L^2}.
\end{multline}
By substituting \eqref{eq:pw2w} and \eqref{eq:psiw2w} into \eqref{eq:wickineq} and taking $M$ sufficiently large we get
\begin{equation}\label{eq:weylineq}
\left\|\left[p\(\sqrt{h}X\)-z\right]^{w}u\right\|_{L^2} + h^{2/3}y^{1/3}\left\|\psi\(
\frac{B \lambda\(\sqrt{h}X\)^2}{y}\)^{w}u\right\|_{L^2}\gtrsim h^{2/3}y^{1/3}\|u\|_{L^2}.
\end{equation}
Now all that remains is to deal with the $\psi$ term, which we will accomplish by showing, with some basic Weyl calculus, 
that it can be absorbed by the other two terms.

Since $\psi$ is real valued $\psi^{w}$ is self-adjoint. Therefore
$$
\left\|\psi\(\frac{B\lambda\(\sqrt{h}X\)^2}{y}\)^{w}u\right\|^{2}_{L^2} = \(\(\psi\(\frac{B\lambda\(\sqrt{h}X\)^2}{y}\)^{w}\)^{2}u, u\).
$$
For the sake of brevity we will henceforth use the notation
\[\Psi\(X\):=\psi\(\frac{B\lambda\(\sqrt{h}X\)^2}{y}\).\]
Lemma \ref{lemma2} can then be rephrased as:
\[\Psi'\(X\)\in S\(\frac{h^{1/2}}{y^{1/2}}\).\]
Let us now recall some basic Weyl calculus.
For symbols $a$ and $b$ in $S\(1\)$, we have the following composition formula for their Weyl quantizations \cite{wicketc},
\begin{equation}\label{eq:wcomp}
a^{w}b^{w}=\(a\# b\)^{w}=\(ab + \frac{1}{4\pi i}\{a,b\} +R\)^{w}, 
\end{equation}
where 
\begin{multline*}
R=-\frac{1}{16\pi^{2}}\int_{0}^{1}\(1-t\)e^{\frac{it}{4\pi}\(D_{\xi}\cdot D_{y}-D_{x}\cdot D_{\eta}\)}\\
\(D_{\xi}\cdot D_{y}-D_{x}\cdot D_{\eta}\)^{2}a\(x,\xi\)b\(y,\eta\)dt \bigg|_{\(y,\eta\)=\(x,\xi\)}.
\end{multline*}
Thus, using that $\left\{\Psi,\Psi\right\}=0$,
\begin{multline*}
\Psi\(X\)\#\Psi\(X\)=\Psi^{2}\(X\)-
\frac{1}{16\pi^{2}}\int_{0}^{1}\(1-t\)e^{\frac{it}{4\pi}
\(D_{\xi}\cdot D_{y}-D_{x}\cdot D_{\eta}\)}\\
\(D_{\xi}\cdot D_{y}-D_{x}\cdot D_{\eta}\)^{2}
\Psi\(x,\xi\)\Psi\(y,\eta\)
dt \bigg|_{\(y,\eta\)=\(x,\xi\)}.
\end{multline*}
By Lemma \ref{lemma2}
$$
\(D_{\xi}\cdot D_{y}-D_{x}\cdot D_{\eta}\)^{2}\Psi\(x,\xi\)
\Psi\(y,\eta\)
=O_{S\(1\)}\(\frac{h}{y}\),
$$
where ``$F_{1} = O_{S\(1\)}\(F_{2}\)$" means $\partial^{\alpha}F_{1}= O\(F_{2}\)$, for all $\alpha$.
By Theorem 4.17 in \cite{zw} the operator $e^{\frac{it}{2}\(D_{\xi}\cdot D_{y}-D_{x}\cdot D_{\eta}\)}$ maps $S\(m\)$
to $S\(m\)$ continuously for any order function $m$, so by the above we get that
$$
\(\Psi\(X\)^{w}\)^{2}=\Psi^{2}\(X\)^{w} + \frac{h}{y}R_{1}^{w},
$$
for some $R_{1}= O_{S\(1\)}\(1\)$.
Thus by applying \eqref{eq:cv} we get
\begin{equation}\label{eq:psicomp}
\left\|\Psi\(X\)^{w}u\right\|^{2}_{L^2} = \(\Psi^2\(X\)^{w}u, u\) + O\(\frac{h}{y}\)\|u\|^{2}_{L^2}.
\end{equation}
To control the first term on the right-hand side we follow a method similar to Lemma 8.2 from \cite{hss}. 
\begin{lemma}\label{lemma3}
$\(\Psi^{2}\(X\)^{w}u, u\) \leq 
\(\(4\frac{\left|p\(\sqrt{h}X\)-z\right|^2}{y^2}\Psi^{2}\(X\)\)^{w}u, u\)
+O\(\frac{h^{1/2}}{y^{1/2}}\)\|u\|^{2}_{L^2}.$
\end{lemma}
\begin{proof}
Recalling \eqref{eq:bval}, we see that on the support of $\Psi\(X\)$ we have that 
\begin{equation}\label{eq:ponpsi}
\left|p\(\sqrt{h}X\)\right|-T\leq \frac{B\lambda\(\sqrt{h}X\)^2}{4} \leq \frac{y}{2}.
\end{equation}
Thus
$$\frac{1}{y}\left|p\(\sqrt{h}X\)-z\right| \geq\frac{1}{y}\(|z|-\left|p\(\sqrt{h}X\)\right|\)$$

$$=\frac{1}{y}\(y+T-\left|p\(\sqrt{h}X\)\right|\)\geq \frac{1}{2},$$
and so
\begin{equation}\label{eq:q}
\Psi^{2}\(X\) \leq
4\frac{\left|p\(\sqrt{h}X\)-z\right|^2}{y^2}\Psi^{2}\(X\).
\end{equation}
Let 
\begin{equation}\label{eq:qdef}
Q\(X\)= 4\frac{\left|p\(\sqrt{h}X\)-z\right|^2}{y^2}\Psi^{2}\(X\)
-\Psi^{2}\(X\)\geq 0.
\end{equation}
By \eqref{eq:pos}, \eqref{eq:ww}, and \eqref{eq:remainder} we get that
\begin{equation}\label{eq:qgard}
\(Q^{w}\(x,D_{x}\)u,u\)_{L^2}\ +\quad 
\end{equation}
\[\left\|\(\int_{0}^{1}\int_{\mathbb{R}^{2n}}\(1-t\)Q''\(X+tY\)Y^{2}e^{-2\pi|Y|^2}2^{n}dY dt\)^{w}u\right\|_{L^2}
\left\|u\right\|_{L^2}\geq 0.\]
To estimate the second term, \eqref{eq:cv} implies that we need to estimate the derivatives of order two and higher of $Q$.

As $|z|>KT+Mh$ and $K>1$,
\[y=|z|-T>\(K-1\)T\gtrsim T.\]
So, for $X$ in the support of $\Psi$, using \eqref{eq:ponpsi}, $y\gtrsim T$, and $y\gtrsim |z|$, we get the following
$$\left| \frac{p\(\sqrt{h}X\)-z}{y}\right|\lesssim \frac{1}{y}\(y+T+|z|\)\lesssim 1. $$
For such $X$, using \eqref{eq:v1'v1} we also have
\begin{equation}\label{eq:p'}
\left|\partial^{\alpha}\frac{p\(\sqrt{h}X\)-z}{y} \right| \lesssim \frac{h^{1/2}}{y}\lambda\(\sqrt{h}X\)
\lesssim \frac{h^{1/2}}{y^{1/2}},\, |\alpha|=1
\end{equation}
and
\begin{equation}\label{eq:p''}
\left| \partial ^{\alpha}\frac{p\(\sqrt{h}X\)-z}{y} \right| \lesssim \frac{h^{|\alpha|/2}}{y},\, |\alpha|\geq 2.
\end{equation}
By the above and \eqref{eq:psisize}, for $|\alpha|\geq 1$,
\[\left| \partial^{\alpha}Q\(X\) \right| \lesssim \frac{h^{1/2}}{y^{1/2}}.\]
Thus by applying the Calder\'{o}n-Vaillancourt theorem \eqref{eq:cv} we can bound the latter term of \eqref{eq:qgard} as follows. 
$$
\left\|\(\int_{0}^{1}\int_{\mathbb{R}^{2n}}\(1-t\)Q''\(X+tY\)Y^{2}e^{-2\pi|Y|^2}2^{n}dY dt\)^{w}u\right\|_{L^2}
\lesssim \frac{h^{1/2}}{y^{1/2}}\|u\|_{L^2}.
$$
Therefore \eqref{eq:qgard} implies a variant of the sharp G\r{a}rding inequality (cf. Theorem 4.32 of \cite{zw}) for $Q$,
$$
\(Q^{w}\(x,D_{x}\)u,u\)_{L^2} + O\(\frac{h^{1/2}}{y^{1/2}}\)\|u\|^{2}_{L^2} \geq 0.
$$
And so by \eqref{eq:qdef} we attain the desired inequality,
$$
\(\Psi^2\(X\)^{w}u, u\) \leq $$
\begin{equation}\label{eq:pelliptic}
\(\(4\frac{|p\(\sqrt{h}X\)-z|^2}{y^2}\Psi^{2}\(X\)\)^{w}u, u\)
+O\(\frac{h^{1/2}}{y^{1/2}}\)\|u\|^{2}_{L^2}.
\end{equation}
\end{proof}
Finally, we have to understand the first term on the right side of \eqref{eq:pelliptic}. 
The estimates \eqref{eq:psisize}, \eqref{eq:p'}, and \eqref{eq:p''} imply that
$$
\partial^{\alpha}\(\frac{\(p\(\sqrt{h}X\)-z\)}{y}\Psi\(X\)\)
=O\(\(\frac{h}{y}\)^{1/2}\),\, |\alpha|\geq 1.
$$
Thus, using this and \eqref{eq:wcomp} and repeating the same Weyl calculus argument used to attain \eqref{eq:psicomp}
we get
\begin{multline*}
4\frac{\left|p\(\sqrt{h}X\)-z\right|^2}{y^2}\Psi^{2}\(X\)\\
= 4 \(\frac{\overline{\(p\(\sqrt{h}X\)-z\)}}{y}\Psi\(X\)\#
\frac{\(p\(\sqrt{h}X\)-z\)}{y}\Psi\(X\)\)\\ -\frac{1}{\pi i}\left\{\frac{\overline{\(p\(\sqrt{h}X\)-z\)}}{y}\Psi\(X\),
\frac{\(p\(\sqrt{h}X\)-z\)}{y}\Psi\(X\) \right\}+
\frac{h}{y}R_{2}\\
=4 \(\frac{\overline{\(p\(\sqrt{h}X\)-z\)}}{y}\Psi\(X\)\#
\frac{\(p\(\sqrt{h}X\)-z\)}{y}\Psi\(X\)\)
+\frac{h}{y}R_{3},
\end{multline*}
where $R_{2}, R_{3}= O_{S\(1\)}\(1\)$.
We also similarly get from \eqref{eq:psisize}, \eqref{eq:p'}, \eqref{eq:p''} and \eqref{eq:wcomp} that
$$
\Psi\(X\)\# \frac{\(p\(\sqrt{h}X\)-z\)}{y} 
= \frac{\(p\(\sqrt{h}X\)-z\)}{y}\Psi\(X\) + \frac{h}{y}R_{4},
$$
for $R_{4}= O_{S\(1\)}\(1\)$. 

Now, using \eqref{eq:psicomp}, Lemma \ref{lemma3}, the fact that $\frac{h}{y}\leq \frac{1}{M}$, and that $\|\Psi^{w}\|_{L^{2}\rightarrow L^{2}}=O\(1\)$,
we can conclude that
$$
\left\Vert\Psi\(X\)^{w}u\right\Vert^{2}_{L^2}\leq \left\Vert \(\Psi\(X\)\)^{w} 
\(\frac{\(p\(\sqrt{h}X\)-z\)}{y}\)^{w}u\right\Vert^{2}_{L^2} + O\(\frac{h^{1/2}}{y^{1/2}}\)\|u\|^2_{L^2}
$$
$$
\lesssim \frac{1}{y^{2}}\left\|\left[p\(\sqrt{h}X\)-z\right]^{w}u\right\|^{2}_{L^2} + O\(\frac{1}{M^{1/2}}\)\|u\|^2_{L^2}.
$$

Plugging this in to \eqref{eq:weylineq} we get
\begin{multline*}
\left\|\left[p\(\sqrt{h}X\)-z\right]^{w}u\right\|_{L^2} + \frac{h^{2/3}}{y^{2/3}}\left\|\left[p\(\sqrt{h}X\)-z\right]^{w}u\right\|_{L^2}\\
+ O\(\frac{1}{M^{1/4}}\)h^{2/3}y^{1/3}\|u\|_{L^2}
\gtrsim h^{2/3}y^{1/3}\|u\|_{L^2}.
\end{multline*}
Then taking $M$ sufficiently large yields
$$
\left\|\left[p\(\sqrt{h}X\)-z\right]^{w}u\right\|_{L^2}\gtrsim h^{2/3}y^{1/3}\|u\|_{L^2}.
$$
Finally, by making the symplectic change of coordinates $x\rightarrow \frac{x}{\sqrt{h}}$, $\xi\rightarrow \sqrt{h}\xi$
we obtain the desired estimate,
$$
\left\|\(p^{w}\(x, hD_{x}\) - z\)u\right\|_{L^2}\gtrsim h^{2/3}y^{1/3}\|u\|_{L^2}.
$$
\end{proof}
\section{From a priori to a resolvent estimate}
Now we will use Theorem \ref{thm1} to prove Theorem \ref{thm2}. To do so it will be convenient to work in the standard, or Kohn-Nirenberg, quantization
rather than the Weyl quantization. In the semiclassical case, this quantization is defined by
$$
a^{KN}\(x, hD_{x}\)u\(x\) = \int_{\mathbb{R}^{2n}} e^{2i\pi\(x-y\)\cdot\xi}a\(x, h\xi\) u\(y\) dy d\xi
$$
$$
=\mathcal{F}^{-1}_{\xi\rightarrow x}a\(x,h\xi\)\mathcal{F}_{y\rightarrow \xi} u\(y\),
$$
where $\mathcal{F}$ denotes the Fourier transform.
Note that just like in the Weyl quantization we have that
$$p^{KN}\(x,hD_{x}\)= -\frac{h^{2}}{4\pi^{2}}\Delta + V\(x\).$$
In this quantization we have the composition formula
\begin{equation}\label{eq:stdcomp}
a^{KN}\(x,hD_{x}\)b^{KN}\(x,hD_{x}\)=\(ab + \frac{h}{2\pi i}D_{\xi}a\cdot D_{x}b +R\)^{KN}\(x,hD_{x}\),
\end{equation}
where 
$$
R=-\frac{h^{2}}{4\pi^{2}}\int_{0}^{1}\(1-t\)e^{\frac{ith}{2\pi}D_{\xi}\cdot D_{y}}
\(D_{\xi}\cdot D_{y}\)^{2}a\(x,\xi\)b\(y,\eta\)dt \bigg|_{\(y,\eta\)=\(x,\xi\)}.
$$
The standard quantization of a symbol is equivalent to the Weyl quantization of a related symbol \cite{zw}, specifically
if $a\in S\(m\)$ for some order function $m$, we have
$$
a^{KN}\(x,hD_{x}\)=\(e^{\frac{h}{4\pi i}\(D_{\xi}\cdot D_{x}\)}a\)^{w}\(x,hD_{x}\)
$$
and 
$$
e^{\frac{h}{4\pi i}\(D_{\xi}\cdot D_{x}\)}a\in S\(m\).
$$
This tells us that some properties of the Weyl quantization can be applied to the standard quantization as well, the 
Calder\'{o}n-Vaillancourt theorem \eqref{eq:cv} among them.
\begin{proof}[Proof of Theorem 2]
To show that $P$, the graph closure of $p^{w}\(x,hD_{x}\)$ on $\Sch\(\R^{n}\)$ has domain $D_{max}=\left\{u\in L^2 : p^{w}u\in L^{2}\right\}$ we follow a method
from H\"{o}rmander found in \cite{horm}. Let $\chi_{\delta}: L^{2}\rightarrow \Sch$ be a family of operators parametrized by $\delta>0$
such that
$\chi_{\delta}u\rightarrow u$ in $L^2$ as $\delta\rightarrow 0$ for all $u\in L^{2}$.
If 
\begin{equation}\label{eq:gc}
\(P\chi_{\delta}-\chi_{\delta}P\)u \rightarrow 0
\end{equation}
in $L^2$ as $\delta\rightarrow 0$ for all $u\in D_{max}$
then we have that $u_{\delta}:=\chi_{\delta}u$ is a sequence of functions in $\Sch$ converging to $u$ and that
$Pu_{\delta}\rightarrow Pu$, thus the domain of $P$ is $D_{max}$.

To accomplish this, let $\phi\in C_{c}^{\infty}\(\R^{n},[0,1]\)$ be a cutoff function with $\phi\(x\)=1$ for $x$ in a neighborhood of $0$.
It suffices to consider the $h=1$ case as $h$ is fixed independent of $\delta$ and thus does not affect issues of convergence.
Then define
\begin{equation*}
    \chi_{\delta}u= \(\phi\(\delta x\)\phi\(\delta \xi\)\)^{KN}u,\quad u\in L^{2}.
\end{equation*}
We then have that $\chi_{\delta}:L^{2}\rightarrow \Sch$ and $\chi_{\delta}u\rightarrow u$ in $L^2$ as $\delta\rightarrow 0$ for all $u\in L^2$
as desired.
We then need to check \eqref{eq:gc}. This can be accomplished using some standard quantization symbol calculus for the commutator 
$[P,\chi_{\delta}]$.
By \eqref{eq:stdcomp} we have
\begin{equation*}
[P,\chi_{\delta}]=\(\frac{1}{2\pi i}\left\{p\(x,\xi\),\phi\(\delta x\)\phi\(\delta \xi\)\right\}+ O_{S\(1\)}\(\delta^{2}\)\)^{KN}
\end{equation*}
\begin{equation}\label{eq:123}
=\frac{\delta}{\pi i}\(\xi\cdot \phi'\(\delta x\)\phi\(\delta \xi\)\)^{KN}
-\frac{\delta}{2\pi i}\(V'\(x\)\cdot \phi'\(\delta \xi\)\phi\(\delta x\)\)^{KN} u +\(O_{S\(1\)}\(\delta^{2}\)\)^{KN}
\end{equation}
$$
=I + II + III.
$$
On the support of $\phi\(\delta x\)\phi\(\delta \xi\)$ we have that $|x|\lesssim \delta^{-1}$ and $|\xi|\lesssim \delta^{-1}$ so,
as $\delta\rightarrow 0$,
$$
\left|\delta \partial^{\alpha} \(\xi\cdot \phi'\(\delta x\)\phi\(\delta \xi\)\)\right|=O\(1\),\quad \forall \alpha
$$
and, recalling \eqref{eq:vbd},
$$
\left|\delta \partial^{\alpha} \(V'\(x\)\cdot \phi'\(\delta \xi\)\phi\(\delta x\)\)\right|=O\(1\),\quad \forall \alpha.
$$
Thus by \eqref{eq:cv}
\[\left\|[P,\chi_{\delta}]\right\|_{L^{2}\rightarrow L^{2}}=O\(1\). \]
It thus suffices to show that $[P,\chi_{\delta}]u\rightarrow 0$ for all $u$ in a dense subset of $L^{2}$.
Term $III$ is easily dealt with because as $\delta\rightarrow 0$, 
\[\|III u\|_{L^2}=O\(\delta^{2}\)\|u\|_{L^2}\rightarrow 0.\]
To deal with terms $I$ and $II$, let $u\in L^{2}$ be such that $\mathcal{F}u\in C^{\infty}_{c}\(\R^{n}\)$.
Then
\[II u= -\frac{\delta}{2\pi i}\phi\(\delta x\) V'\(x\)\cdot \mathcal{F}^{-1}\(\phi'\(\delta \xi\)\(\mathcal{F}u\)\(\xi\)\). \]
Note that $\phi'\(\delta \xi\)$ is supported where $|\xi|\sim\delta^{-1}$ so for $\delta$ sufficiently small
\[\phi'\(\delta \xi\)\(\mathcal{F}u\)\(\xi\)=0\]
and so
\[\|II u\|_{L^2}\rightarrow 0.\]
Also, 
\[I u=\frac{\delta}{\pi i}\phi'\(\delta x\)\cdot \mathcal{F}^{-1} \(\xi \phi\(\delta \xi\)\(\mathcal{F}u\)\(\xi\)\).\]
Because $\mathcal{F}u\(\xi\)$ is compactly supported and $\phi=1$ in a neighborhood of $0$, for $\delta$ sufficiently small we have
\[\phi\(\delta \xi\)\(\mathcal{F}u\)\(\xi\)=\(\mathcal{F}u\)\(\xi\).\]
And then
\[I u=\frac{\delta}{\pi i}\phi'\(\delta x\)\cdot \mathcal{F}^{-1}\( \xi \(\mathcal{F}u\)\(\xi\)\)\]
\[=-\frac{\delta}{2\pi^{2}}\phi'\(\delta x\)\cdot u'\(x\).\]
Since $\mathcal{F}u\in C_{c}^{\infty}$ we have $u'\in L^{2}$ so 
\[\|I u\|_{L^2}\rightarrow 0.\]

Therefore \eqref{eq:gc} holds, which tells us that the graph closure of $p^{w}\(x,hD_{x}\)$ on $\Sch$, 
has the domain $D_{max}$. Thus, for $z$ and $h$ satisfying the conditions in Theorem \ref{thm1} we have
\[\|\(P-z\)u\|_{L^{2}}\gtrsim h^{2/3}\(|z|-T\)^{1/3}\|u\|_{L^{2}}\quad \forall u\in D_{max}.\] 
We thus have that $P-z$ is injective on $D_{max}$ and has closed range. 
We can apply the same argument to the formal adjoint of $p^{w}$ on $\Sch$, $\overline{p}^{w}-\overline{z}=\(|\xi|^{2}+\overline{V}\(x\)\)^{w}-\overline{z}$,
and we similarly get its graph closure is
$\overline{P}-\overline{z}=-\frac{h^{2}}{4\pi^{2}}\Delta +\overline{V}\(x\) -\overline{z}$ with domain $\{u\in L^{2}: \overline{p}^{w}u\in L^{2}\}$,
which is also injective with closed range. As $\overline{P}-\overline{z}$ has maximal domain
we have that $\overline{P}-\overline{z}=\(P-z\)^{*}$. Thus $P-z$ is invertible, and we get the desired resolvent estimate,
\[\|\(P-z\)^{-1}u\|_{L^{2}}\lesssim h^{-2/3}\(|z|-T\)^{-1/3}\|u\|_{L^2}.\]
\end{proof}.

\section{The bounded $z$ case}
In the preceding sections, the condition that was placed on $V$ in \eqref{eq:vv't}, that
\[ |V_{2}\(x\)|-T\lesssim \left|V_{2}'\(x\)\right|^{2},\quad \forall x\in \R^{n}, \]
is only used once. It is used so that we can get the inequality \eqref{eq:bval}
\[|p\(X\)|-T\leq \frac{B \lambda\(X\)^{2}}{4},\quad \forall X\in \R^{2n},\]
which implies
\[\psi^{2}\(\frac{B\lambda\(\sqrt{h}X\)^{2}}{y}\)\leq 
4\frac{\left|p\(\sqrt{h}X\)-z\right|^{2}}{y^{2}}\psi^{2}\(\frac{B\lambda\(\sqrt{h}X\)^{2}}{y}\)
\lesssim 1. \]
We see that the condition on $V_{2}$ in \eqref{eq:vv't} is only needed in the region where $\lambda^{2}\lesssim y$.
Thus if we only consider values of $z$ such that $|z| - T\leq R$ for some $R>0$ we do not need this condition on $V_{2}$ to apply globally.
Instead we need there to exist some constant $L>0$ such that
\begin{equation}\label{eq:altcon}
|V_{2}\(x\)|-T\lesssim \left|V_{2}'\(x\)\right|^{2},\quad \forall x\in \left\{x\in \R^{n} : |V_{2}'\(x\)|\leq L \right\} .
\end{equation} 
Then by taking $B$ large enough (and hence $A$ small enough), we still get
\[\psi^{2}\(\frac{B\lambda\(\sqrt{h}X\)^{2}}{y}\)\leq 4\frac{\left|p\(\sqrt{h}X\)-z\right|^{2}}{y^{2}}\psi^{2}\(\frac{B\lambda\(\sqrt{h}X\)^{2}}{y}\)\lesssim 1. \]
The rest of the proof can remain unchanged. This results in the following.
\begin{theorem}
Let $p$ be in $C^{\infty}\(\R^{2n}\)$ be given by  $p=\xi^{2}+V\(x\)$ with $V=V_{1}+iV_{2}$, $V_{1}, V_{2}$ real valued, $V_{1}\geq 0$, 
$V''\in S\(1\)$, and
\[ V_{2}\(x\)-T \lesssim \left|V_{2}'\(x\)\right|^{2},\quad \forall x\in \left\{x\in \R^{n} : |V_{2}'\(x\)|\leq L \right\},  \]
for some $L>0$, $T\geq 0$.
Then for any $R>0$, $K>1$, there exist positive constants $A$, $M$, and $h_{0}$ such that for all $0<h\leq h_{0}$ and $z\in \C$ with $Mh\leq |z| -KT \leq R$ and $\rr z\leq Ah^{2/3}\(|z| -T\)^{1/3}$ we have
\[ \left\|\left[p^{w}\(x,hD_{x}\)-z\right]u\right\|_{L^2} \gtrsim h^{2/3}\(|z| -T\)^{1/3}\left\|u\right\|_{L^2},\quad \forall u\in \Sch\(\R^{n}\), \]
and taking $P$, the $L^2$-graph closure of $p^{w}$ on $\Sch$, we have
\[ \left\|\(P-z\)^{-1}u\right\|_{L^2}\lesssim h^{-2/3}\(|z|- T\)^{-1/3}\|u\|_{L^2},\quad \forall u\in L^{2}. \]
\end{theorem}
The set of potentials $V$ to which this can apply is very broad. Provided $V_{1}\geq 0$ and $V'' \in S\(1\)$, then \eqref{eq:altcon} will be satisfied for some $T$ and $L$
as long as there is no sequence of points $x_{j}$ along which $\left|V_{2}'\(x_{j}\)\right|\rightarrow 0$ and $V_{2}\(x_{j}\)\rightarrow \infty$.\\ \\
\textit{Acknowledgement.} I am very thankful to Michael Hitrik, who provided substantial guidance and input in the creation of this paper.

\end{document}